\documentclass[12pt]{amsart}
\usepackage{amsmath,amssymb,amscd,amsthm, enumerate}

\def\F{\mathbb F}
\def\Z{\mathbb Z}
\def\Q{\mathbb Q}
\def\AF{\overline{\F}}
\def\P{\mathbb{P}}

\def\T{\textrm{Tr}}

\def\Gal{\textrm{Gal}}
\def\L{\mathcal{L}}

\theoremstyle{plain}

\newtheorem{theorem}{Theorem}[section]

\newtheorem{cor}[theorem]{Corollary}
\newtheorem{ques}[theorem]{Question}

\author{Amanda Knecht}
\address{Department of Mathematics and Statistics\\
Villanova University \\
Villanova, PA 19085}
\email{Amanda.Knecht@villanova.edu}

\title[Degree of Unirationality for del Pezzo Surfaces over $F_q$]{Degree of Unirationality for del Pezzo Surfaces over Finite Fields}
\begin{document}
  \maketitle
  \begin{abstract}   We address the question of the degree of unirational parameterizations of degree four and degree three del Pezzo  surfaces.  Specifically we show that degree four del Pezzo surfaces over finite fields admit degree two parameterizations and minimal cubic surfaces admit parameterizations of degree 6.  It is an open question whether or not  minimal cubic surfaces over finite fields can admit degree 3 or 4 parameterizations.\end{abstract}
  \section{ Introduction}

It is a classical result that for every cubic surface $S_3$ defined over an algebraically closed field there exists a degree one rational map $\P^2 \dashrightarrow S_3$.   We say that such a surface is \textit{rational}.   But over a non-algebraically closed field there are many examples of non-rational cubic surfaces.   A surface $S_3$ is \textit{unirational} if there exists a finite to one rational map $\P^2 \dashrightarrow S_3$.
 In 1943 Segre proved that a smooth cubic surface defined  over $\Q$ is unirational if and only if it contains a rational point \cite{MR0009471}.     Manin then showed that the same is true for cubic surfaces over finite fields containing at least  34 elements \cite{MR833513}.    More recently Koll\'{a}r proved that over an arbitrary field a cubic hypersurface with a rational point is always unirational \cite{MR1956057}.  
 Cubic surfaces over finite fields always have points \cite{Chevalley} so are unirational.
 The aim of this note is to describe the possible degrees of the unirational maps to non-rational cubic surfaces over finite fields.  Let $\F_q$ be a finite field of size $q=p^r$.

\begin{theorem}\label{thm1}  Let $S_3$ be a non-rational cubic surface defined over a finite field $\F_q$.  If there exists a degree two rational map $\P^2 \dashrightarrow S_3$ and $q\neq 2$, then $S_3$ contains a line defined over $\F_q$.  If $S_3$ does not contain a line defined over $\F_q$, then there exists a degree 6 unirational parameterization of $S_3$.
\end{theorem}

  Fix an algebraic closure  $\AF_q$, let  $G=$ Gal $(\AF_q, \F_q)$ and   $X$ be a smooth  surface defined over $\F_q.$  Let $N(X)=$ Pic $(X \otimes \AF_q)$.
In his book \textit{Cubic Forms}, Manin proves that there is always a degree six parameterization of a cubic surface over large enough fields of characteristic different from two.   He suggests that his lower bound of 35 for the size of the field may not be optimal but gives an example over $\F_4$  where his proof falls apart.   He also proves  that if there exist a rational map of finite degree $\varphi: \P^2 \dashrightarrow X$, then the degree of $\varphi$ is divisible by the least common multiple of the exponents of the groups $H^1(H,N(X))$ for all possible closed subgroups $H \subseteq G$ \cite{MR833513}.

By Theorem \ref{thm1}, we know that a minimal cubic surface never has a rational parameterization of degree 2 but always has one of degree 6.    The 27 lines on a cubic surface are acted on by the Weyl group of $E_6$, $W(E_6)$.   There are twenty-five conjugacy classes in this group which Frame denotes $C_1, \ldots, C_{25}$ \cite{MR0047038}.  For each class $C_i$, let $H$ be the cyclic subgroup in $W(E_6)$ generated by some element of the class.  Manin calculates  $H^1(H, N(S_3))$ for each class and places them in a table on pages 176-177 of the book \textit{Cubic Forms}   \cite{MR833513}.  There are two mistakes in his table.  In Corollary 1.17 Urabe proved that the numbers $h^1$ are all square \cite{MR1322894}, so the $H^1$ for classes $C_4$ and $C_{20}$ are 0 instead of $\Z_2$.  Noting this correction to the chart, we find that the only possible nonzero $H^1$ groups are $\Z_2 \times \Z_2$ and $ \Z_3 \times \Z_3$.  A correct table can be found in a paper by Shuijing Li \cite{0904.3555}.

There are two questions left open concerning the degree of the parameterization:

\begin{ques}If $S_3$ is a minimal cubic surface such that $$H^1(\Gal(\AF_q,K),N(S_3))= \Z_2 \times \Z_2$$ for some algebraic  extension $K$ of $\F_q$, does there exist a degree 4 unirational parameterization of $S_3$?
\end{ques}
 
\begin{ques} If $S_3$ is a minimal cubic surface such that $$H^1(Gal(\AF_q,K),N(S_3))= \Z_3 \times \Z_3$$ for some algebraic  extension $K$ of $\F_q$, does there exist a degree 3 unirational parameterization of $S_3$?
  \end{ques}
  
It should be noted that if $S$ is a degree $d$ del Pezzo surface  defined over a finite field $k$, then $S$ is rational when $d\geq 5$ and unirational of degree two when $d=4$ and $|k|>22$ \cite{MR0225780} \cite{MR833513}.  In the following section we extend the result for degree four del Pezzo surfaces to all finite fields.
There has also been recent progress in the study of degree 2 del Pezzo surfaces over finite fields. Salgado, Testa, and V‡rilly-Alvarado prove that degree two del Pezzo surfaces over finite fields are unirational except for possibly three exceptional cases   \cite{1304.6798}.\\

\noindent \textbf{Acknowledgments:} I am grateful to Brendan Hassett for the many conversations we had about this topic.
\section{Degree 2 Maps for degree 4 del Pezzo Surfaces}
Over an algebraically closed field, a degree four del Pezzo surface $S_4$ is the blow up of the projective plane at five points, no three of which are collinear.   Such a surface may also be thought of as a rational conic bundle with four singular fibers.   The four singular fibers consist of pairs of exceptional curves intersecting transversely at a point.    Manin showed that  a degree four del Pezzo surface is unirational of degree two over a finite field $k$  as long as $|k|>22$ \cite{MR833513}.  The field is required to contain 22 elements because his rational map uses a $k$-rational point on the surface that is not on any exceptional curves.  Points on degree four del Pezzo surfaces can be contained in at most two exceptional curves.    Below we give a proof that there is a degree two rational map over any finite field.

\begin{theorem}\label{Thm2}  Suppose $S_4$ is a degree four del Pezzo surface defined over a finite field $k$.  Then there exists a unirational parameterization of $S_4$ of degree 2 defined over $k$.
\end{theorem}
\begin{proof}
There are three cases to consider.
\begin{enumerate}[(i)]
\item There exists a point $x\in S_4(k)$ not contained in an exceptional line of $S_4$.
\item All the points in $S_4(k)$ are contained in exceptional lines but there is a point contained in exactly one exceptional line.
\item  All $k$-rational points are the intersection of two exceptional curves.  
\end{enumerate}
The first case is proven in Manin's book.  
In the second case, suppose we have a $k$-rational point $x$ contained on exactly one line on the surface.  Then this line is defined over $k$ and we can blow it down to produce a birational map to a degree 5 del Pezzo surface.  These are known to be rational. 

Finally, if $x$ is contained in two exceptional curves $L_1$ and $L_2$, then their union is defined over a quadratic extension $k'$ of $k$.  There are two other lines $L_3, L_4$ defined over $k'$ such that $(L_1, L_3)=(L_2, L_4)= (L_3, L_4)=1$.  Each of these pairs of lines yields a conic bundle with four degenerate fibers.  A smooth fibre in one bundle is a bisection of the other.  As long as one of the bundles has a smooth fibre, there is a degree two unirational parameterization of $S_4$ defined over $k$.  So we suppose by way of contradiction that there are not any smooth fibers in either bundle.  Then each bundle has exactly four $k$-rational fibers; each is a union of two lines.  Over a field of size $q$, the projective line has $q+1$ points.  Thus, the field $S_4$ must be defined over is 
%either $\F_2$ or 
$\F_3$.  
%If $k=\F_2$ then there are 3 singular fibers in each bundle and $S_4$ has 6 points over $k$.  But the ChevalleyÐWarning theorem tells us that the number of points on a geometrically rational surface over a field of characteristic $p$ is congruent to 1 modulo $p$ \cite{Chevalley}\cite{Warning}.  Thus, $k \neq \F_2$.
 The Chevalley-Warning theorem tells us that the number of points on a geometrically rational surface over a field of characteristic $p$ is congruent to 1 modulo $p$ \cite{Chevalley}\cite{Warning}. If $k=\F_3$, then there are 8 $k$-rational points on $S_4$ but 8 is congruent to 2 modulo 3.  
\end{proof}
A corollary of this result is the following fact that seems to be known to experts for any smooth cubic hypersurface containing a line.
\begin{cor}  Suppose $S_3$ is a cubic surface defined over a field $k$.  Suppose further that $S_3$ contains a line defined over $k$.  Then there exists a unirational parameterization of $S_3$ of degree 2 defined over $k$.
\end{cor}
\begin{proof}
If $S_3$ contains a line defined over $k$, we can blow the line down to get a birational map from $S_3$ to a degree four del Pezzo surface.  Then we apply Theorem \ref{Thm2}.
\end{proof}

  \section{Degree 2 Maps for Cubic Surfaces}
  We saw in the previous section that a cubic surface $S_3$ defined over a finite field $k$ will have a degree two unirational parameterization if the surface contains a line defined over $k$.  In this section we show that having a line defined over $k$ is actually a necessary condition for  being unirational of degree two when the field is of odd characteristic.  This result is a direct consequence of the following theorem of Bayle and Beauville.

Let $k$ be an algebraically closed field of odd characteristic and $S$ a smooth, projective, connected surface over $k$.  Also let $\sigma$ be a biregular involution of $S$.  We say $(S, \sigma)$ is minimal if each exceptional curve $E$ on $S$ satisfies $\sigma E \neq E$ and $E \cap  \sigma E = \emptyset$.
  \begin{theorem}[Bayle, Beauville  \cite{MR1802909}]
  Let $(S,\sigma)$ be a minimal pair. One of the following holds: 
  \begin{itemize}
  \item [(i)] There exists a smooth $\P^1$-fibration $f : S \rightarrow \P^1$ and a non-trivial involution
$\tau$ of $\P^1$ such that $f \circ \sigma = \tau \circ f$
\item[(ii)] There exists a fibration $f:S \rightarrow \P^1$ such that $f \circ \sigma = f$; the smooth fibers of
$f$ are rational curves, on which $\sigma$ induces a non-trivial involution; any singular fibre is the union of two rational curves exchanged by $\sigma$, meeting at one point.
\item[(iii)] $S$ is isomorphic to $\P^2$  with linear involution $\sigma$.
\item[(iv)] $S$ is isomorphic to $\P^1 \times \P^1$ with the involution $(x, y) \mapsto (y,x)$. 
\item[(v)] $S$ is a Del Pezzo surface of degree 2 and $\sigma$ the Geiser involution. 
\item[(vi)] $S$ is a Del Pezzo surface of degree 1 and $\sigma$ the Bertini involution.
\end{itemize} 
  \end{theorem}
  
Suppose $S_3$ is a smooth, projective, minimal cubic surface that admits a degree two unirational parameterization.  Then there exists another smooth, projective, minimal  rational surface $S$ with an involution $\sigma$ such that $T=S/<\sigma>$ is birational to $S_3$.  A minimal cubic surface is never birational to a conic bundle, so cases (i) and (ii) above are ruled out for $S$.  We are left to check cases (iii)-(vi).
 
 Consider the quotient $q: S \rightarrow T$ and decompose the fixed locus of $\sigma$ as a union of isolated points $p_i$ and disjoint curves $B_{\ell}$.  Notice that the self intersection of $B_{\ell}$ is never -1 since $S$ is minimal.  The pullback of the anticanonical bundle of $T$ is the sum of an ample and nef divisors, $q^{*}(-K_T)= -K_S + \sum B_{\ell}$.  Thus the surface $T$ is a del Pezzo.  We can easily compute 
 \begin{equation}\label{canon}
 2K_T^2 = K_S^2 - 2K_S\sum B_{\ell} + \sum B_{\ell}^2.
 \end{equation}  Notice that in cases (iii)-(iv) above, $-K_S$ is ample, so all terms in the sum above are positive. 
If $T$ is a cubic surface then equation $(\ref{canon})$ becomes $6= K_S^2 - 2K_S\sum B_{\ell} + \sum B_{\ell}^2.$  Since the terms in the sum are positive, S can only be a degree 2 or 1 del Pezzo surface, (cases (iii) and (iv) have $K_S^2 = 9$ and $8$, respectively).

 Let us first consider the case $K_S^2=2$. Equation $(\ref{canon})$ becomes $4= \sum B_{\ell}^2 - 2K_S \sum B_{\ell}$.  Because $-K_S$ is ample, the  only  possibilities for $- K_S \sum B_{\ell}$  are 2 and 1.    If $- K_S \sum B_{\ell}=2$, then $\sum B_{\ell}^2=0$.  Calculations of the intersection numbers of curves in $S$ show that in this case there is only one curve $B_1$ fixed by $\sigma$ and that curve is a conic passing through four of the exceptional points of $S$.  But then that means there is a conic on the cubic surfaces $T$ and $S_3$.  Finally we arrive at a contradiction to the minimality of $S_3$ since a cubic surface containing a conic also contains a line coplanar to that conic.    
If $- K_S \sum B_{\ell}=1$, then $\sum B_{\ell}^2=2$.  This time calculations of the intersection numbers of curves in $S$ show that this case is not possible.  Thus, S cannot be a degree two del Pezzo surface.

Next we consider the case $K_S^2=1$. Equation $(\ref{canon})$ becomes $5= \sum B_{\ell}^2 - 2K_S \sum B_{\ell}$.  Just as in the case above, the  only  possibilities for $- K_S \sum B_{\ell}$  are 2 and 1.    If $- K_S \sum B_{\ell}=2$, then $\sum B_{\ell}^2=1$.     If $- K_S \sum B_{\ell}=1$, then $\sum B_{\ell}^2=3$.  Intersection number computations show that both cases are not possible.  So $S$ is not a degree 1 del Pezzo, and we have proven the following theorem.

\begin{theorem}  If $S_3$ is a minimal cubic surface defined over a field of odd characteristic, then there does not exist a unirational parameterization of $X$ of degree 2.  
\end{theorem}

  \section{Degree 6 Maps for  Cubic Surfaces}
  We start this section by recalling Manin's construction of a degree six unirational parameterization for cubic surfaces over larger enough finite fields.  
  \begin{theorem}[Y. Manin]\label{ThmM}
  Let $S_3$ be a cubic surface defined over a finite field $k$ and suppose that there is a point $p\in S_3(k)$ which is not on an exceptional curve.  Then there exists a  rational map $\varphi: \P^2 \dashrightarrow S_3$ of degree $6$.
  \end{theorem}
Manin goes on to state that a cubic surface defined over a field with at least 34 elements, will contain a point not on an exceptional curve.  We can lower that bound to five.

\begin{theorem}
Let $S_3$ be a minimal cubic surface defined over $\F_q$ with  $q\geq 5$.  Then there exists a point $x\in S_3(\F_q)$ not contained on an exceptional curve.\end{theorem}
\begin{proof}
Let $S_3$ be a minimal cubic surface. Suppose all the points defined over $\F_q$ lie on exceptional curves and let $x\in S_3(\F_q)$.  If $x$ is contained in exactly one exceptional line, then the line is defined over $\F_q$.    If $x$ is contained in exactly two exceptional curves, they must be Galois conjugates defined over $\F_{q^2}$.  Consider the plane $P$ spanned by the two lines.  The intersection of the plane and the surface is three lines.  The third line must be defined over the ground field.  In both cases the surfaces in not minimal and contradicts our hypothesis.

Thus, $x$ must be contained in three exceptional curves and is called an Eckardt point.  A cubic surface defined over a field of odd characteristic can have 1, 2, 3, 4, 6, 9, 10 or 18 Eckardt points \cite{MR0233272}.  Over a field of characteristic two the possibilities are 1, 3, 5, 9, 13 and 45 \cite{MR0233272}.  
Let $q=p^r$ and fix an algebraic closure  $\AF_q$, let  $G=$ Gal $(\AF_q, \F_q)$ and $F\in G$  the Frobenius automorphism sending an element $z$ to $z^p$.    Let $N(S_3)=$ Pic $(S_3 \otimes \AF_q)$ and $F^*: N(S_3) \rightarrow N(S_3)$ be the action of Frobenius on the Picard group.  
Weil proved that the following formula for the number of $\F_q$ rational points on a cubic surface \cite{ MR0092196};
$$|S_3(\F_q)|= q^2+ q \; \T F^*+1. $$ 
If the surface $S_3$ is minimal, then the possibilities for $\T F^*$ are $-2, -1, 0, 1$ and $2$ \cite{MR833513}.  A simple computation shows that the only fields in which all points on the surface can be Eckardt points are $\F_2, \F_3$ and $\F_4$.  
\end{proof}
\begin{cor}
Let $S_3$ be a cubic surface defined over a finite field with at least five elements.  Then the minimal degree of a unirational parameterization of $S_3$ is at most six.
\end{cor}
\begin{proof}
If $S_3$ is minimal, there is a point not contained on the exceptional curves, and Theorem  \ref{ThmM}  gives a degree six unirational map.  If $S_3$ is not minimal, then Theorem  \ref{Thm2}  gives a degree two unirational map. 
\end{proof}

Swinnerton-Dyer  classified cubic surfaces not containing a line and having only Eckardt points as rational points \cite{MR619244}.   He proved that these surfaces only exist when $q=2$ or $4$.  

Over the field $\F_2$ there is a unique cubic surface with one point.  The surface is defined by the equation,
\begin{equation}\label{C1}
 Y^3+Y^2Z+Z^3+W(X^2+Y^2+YZ+Z^2)+XW^2+W^3=0.
 \end{equation} The surface contains only one point $[1,0,0,0]$ over $\F_2$, and that point is an Eckardt point.  The lines on the surface are all defined over the degree three extension $\F_8$ where all 121 points of the surface are contained on the lines and thirteen of the points are Eckardt points.   Also over the field of two elements there is a unique cubic surface with three Eckardt points as its rational points.  This surfaces is given by the equation,
\begin{equation}\label{C2}XY(X+Y)+Z^3+Z^2W+W^3=0.\end{equation}  It is a rational surface, but splits over $\F_{64}$.   Fifteen of the lines are defined over $\F_8$ and the other twelve over $\F_{64}$.

Swinnerton-Dyer found two inequivalent surfaces over $\F_4$.  If we let the elements of $\F_4$ be $0,1,\alpha$ and $\alpha+1$, where $\alpha^2 + \alpha + 1=0$, then the equations of the two surfaces are,
\begin{equation}\label{C3} X^3+Y^3+Z^3 +\theta W^3=0 \textrm{ where } \theta= \alpha \textrm{ or }  \alpha+1.\end{equation}  These surfaces each have nine rational points over $\F_4$, and they are all Eckardt points.   The lines on these surfaces are defined over the degree six extension $\F_{64}$.

In order to find a degree six unirational parameterization of these four surfaces, we consider points over degree two extensions instead of just points over the ground field.  Koll\'ar gave the following degree six unirational map \cite{MR1956057}.
  \begin{theorem}[J. Koll\'ar]\label{thmK}
  Let $S_3$ be a cubic surface defined over a finite field $\F_q$.  Fix a point $x\in S_3(\F_q)$ and a line $L \in \P^3$ not contained in the surface. Then $L\cap S_3 = \{x, s, s'\}$ where $s,s' \in S_3(\F_{q^2})$ are Galois conjugate points.  If there exists an $x\in S_3(\F_q)$ such that $s$ and $s'$ do not lie on the exceptional curves of $S_3$, then  there exists a  dominant rational map $\varphi: \P^2 \dashrightarrow S_3$.  
  \end{theorem}
%  \begin{proof}
%Let $\{x, s, s'\}$ be as in the statement of the theorem, where $s,s' \in S_3(\F_{q^2})$ are Galois conjugate points. Let $C_s= T_sS_3 \cap S_3$ and $C_{s'}=T_{s'}S_3\cap S_3$ be  cubic curves with singularities at $s$ and $s'$, respectively.
%Since  $C_s$ and $C_{s'}$ are Galois conjugate rational cubic curves, $C_s \times C_{s'}$ is birational to $\P^2$ over $\F_q$.  We define the third point map  over $\F_q$,
%$\varphi: C_s \times C_{s'}\dashrightarrow S_3$, where $\varphi(y,y')$ is the the third intersection point on the surface $S_3$ when you draw a line between $y$ and $y'$.  This map is dominant as long as $C_s$ and $C_{s'}$ do not lie on the same plane.  But, we know they do not because if they did that plane would be the tangent plane and we would have an irreducible cubic curve with two nodes.  
%\end{proof}

  We will assume that the surface is minimal, since we have already seen how to produce a degree two parameterization for non-minimal surfaces.  We can  see that the degree of the 
 the third point map is  six as follows.
 
Pick $z\in S_3$ such that $z\notin C_{s} \cup C_{s'}$.  Let $\pi_z : \P^3 \backslash \{z\} \dashrightarrow T_{s'}S_3$ be the projection from $Z$ onto the tangent plane of $S_3$ at $s'$.  Since $\pi_z(C_s)$ is a cubic curve in $T_{s'}S_3$, it will intersect $C_{s'}$ in nine points, counting multiplicity.   Three of the points lie on the line $\L= T_sS_3 \cap T_{s'}S_3$,  and they are the intersection of our two rational curves and are fixed by $\pi_z$.  Since $\L$ is defined over $\F_q$ and $S_3$ is minimal, $\L$ does not lie on $S_3$.  It should be noted that $p \notin \L$ because if it were, $p$ would also be contained in $T_sS_3$.  Then $\L$ and $s'$ would be contained in $T_sS_3$, and thus $T_sS_3=T_{s'}S_3$.  Since $\L$ does not contain any $\F_q$ points we know that $\L \cap S_3$ consists of three $\F_{q^3}$ points on $C_s$ and $C_{s'}$.  Since the tangent lines to these points lie in different planes, we know that $C_s$ and $C_{s'}$ intersect transversely.  It would take a very special $z$ to introduce multiplicity after projection.  So generically there are three points of intersect all of multiplicity one.  These three points are in the indeterminacy locus of the map $\varphi$ since they are on both $C_s$ and $C_{s'}$, so we do not count them when counting the degree of $\varphi$.    Thus the degree of $\varphi$ is six.  

In order to finish the proof of  Theorem \ref{thm1}, we must show that the exceptional cases above satisfy the assumptions of Theorem \ref{thmK}.  These surfaces defined by equations (\ref{C1}) and (\ref{C3}) are in the conjugacy class $C_{11}$ of Frame \cite{MR0047038},  which means they split over a degree three extension and contain no lines over $\F_q$ or $\F_{q^2}$.  Over a degree two extension, they contain $q^4-2q^2+1$ points with $q^2-2q+1$ of them lying on exceptional lines.
Thus there are $q(q+2)(q-1)^2$ $\F_{q^2}$ points on the surface away from the exceptional locus.  So the assumptions of Koll\'ar's theorem are met, and there exists a degree six unirational parameterization.

The surface defined by equation (\ref{C2}) is in Frame's conjugacy class $C_{22}$.  As noted above, the surface is rational.

\bibliographystyle{alpha}
\bibliography{CubicBib}
    \end{document}